\documentclass[reqno,12pt,a4paper]{amsart}

\usepackage{euscript}
\usepackage{multicol}
\usepackage{amsmath}

\usepackage{amsmath,amsfonts,amsthm,amssymb,amsxtra}
\usepackage{mathrsfs}
\usepackage{amssymb}
\usepackage{amsfonts}
\usepackage{latexsym}
\usepackage{amsthm}

\usepackage{mathtools,xparse}

\DeclarePairedDelimiter{\oldnormaux}{\bracevert}{\bracevert}

\NewDocumentCommand{\oldnorm}{som}{%
  \IfBooleanTF{#1}
    {\oldnormaux*{#3}}
    {\IfNoValueTF{#2}
       {\oldnormaux*{\vphantom{dq}#3}}
       {\oldnormaux[#2]{#3}}%
    }%
}

\theoremstyle{plain}
\newtheorem{theorem}{\bf Theorem}[section]

\newtheorem{corollary}[theorem]{\bf Corollary}
\newtheorem{proposition}[theorem]{\bf Proposition}

\def\l{\lambda}

\def\t{\tau}

\usepackage{pgfplots}
 
\pgfplotsset{compat = newest}

\numberwithin{equation}{section}

\let\ge\geqslant
\let\le\leqslant
\let\geq\geqslant
\let\leq\leqslant
\newcommand{\ca}{\begin{cases}}
\newcommand{\ac}{\end{cases}}
\newcommand{\ma}{\begin{pmatrix}}
\newcommand{\am}{\end{pmatrix}}
\renewcommand{\[}{\begin{equation}}
\renewcommand{\]}{\end{equation}}
\def\eq{\begin{equation}}
\def\qe{\end{equation}}
\def\[{\begin{equation}}

\title[Discrete Spectrum of the Bilayer Graphene Operator]{Discrete Spectrum of the Bilayer Graphene Operator}
\author{Siyu Gao, and Oleg Safronov}

\address{Department of Mathematics and Statistics, UNCC, 9201 University City Blvd., Charlotte, NC 28223, USA}

\email{sgao7@charlotte.edu}

\email{osafrono@charlotte.edu}

\begin{document}
\maketitle

\thispagestyle{empty}

\begin{abstract}
We consider the graphene operator $D_m$ perturbed by a decaying potential $\alpha V$, where $\alpha$ is a coupling constant. We study the number $N(\lambda,\alpha)$ of eigenvalues of the operator $D(t)=D_m-tV$ passing through a regular point $\lambda\in\rho(D_m)$ as $t$ changes from $0$ to $\alpha$. We obtain asymptotic formulas for $N(\lambda,\alpha)$ as $\alpha\to\infty$.
\end{abstract}

\section{Introduction}

The free bilayer graphene operator $D_m$ is defined on the space $L^2(\mathbb R^2,\mathbb C^2)$ by
\begin{eqnarray*}
    D_m=\begin{pmatrix}
                m&\left(\frac{\partial}{\partial x_1}-i\frac{\partial}{\partial x_2}\right)^2\\\left(\frac{\partial}{\partial x_1}+i\frac{\partial}{\partial x_2}\right)^2&-m
                \end{pmatrix},\quad\text{where }m>0.
\end{eqnarray*} 
While  this particular  operator  was   introduced by Ferrulli,  Laptev and Safronov in \cite{SLF}, 
the study of bilayer graphene and its properties,  including the development of theoretical models and operators describing its electronic structure, 
has been ongoing for a longer period of time (since 2004).

The spectrum of the operator $D_m$ coincides with the set $(-\infty,-m]\cup[m,\infty)$. Thus, the interval $(-m,m)$ is a gap in the spectrum. Let $V:\mathbb R^2\to[0,\infty)$ be a non-negative potential on that decays as $|x|\to\infty$. In this paper, we analyze the discrete spectrum of the perturbed operator
\begin{eqnarray}
\label{main operator}
D(\alpha)=D_m-\alpha V,\quad\alpha>0.
\end{eqnarray} The operator $V$ in \eqref{main operator} is understood as the matrix operator
\begin{eqnarray*}
    V\cdot I=\begin{pmatrix}
        V&0\\0&V
    \end{pmatrix}.
\end{eqnarray*} The constraints that we impose on the potential $V$ guarantee that the spectrum of $D(\alpha)$ is discrete in $(-m,m)$. That implies that it consists of isolated eigenvalues of finite multiplicity, which move monotonically from the right to the left with the growth of $\alpha$.

Let us fix a point $\lambda\in(-m,m)$ and introduce $N(\lambda,\alpha)$ as the number of eigenvalues of $D(t)$ passing through $\lambda$ when $t$ increases from $0$ to $\alpha$. We investigate the asymptotic behavior of $N(\lambda,\alpha)$ for sufficiently large values of $\alpha$. 

We define the cubes $\mathbb Q_n=[0,1)^2+n$ with $n\in\mathbb Z^2$, and set 
\begin{align*}
    \|V\|^q_{L^q(\mathbb Q_n)}=\int_{\mathbb Q_n}|V|^q\,dx,\qquad\text{for }q>0.
\end{align*} Our first result deals with the case
\begin{eqnarray}
\label{conditions for the first case}
    \quad\sum_{n\in\mathbb Z^2}\|V\|_{L^q(\mathbb Q_n)}<\infty,\qquad q>1.
\end{eqnarray}
\begin{theorem}
\label{first theorem}    
Let $V$ satisfy \eqref{conditions for the first case}. Then
    \begin{eqnarray*}
        N(\lambda,\alpha)\sim\frac{\alpha}{4\pi}\int_{\mathbb R^2}V\,dx,\quad\text{as }\alpha\to\infty.
    \end{eqnarray*}
\end{theorem}

Now we consider the case where $V\not\in L^1(\mathbb R^2)$. Instead, we assume that $V$ is a bounded function obeying the condition
\begin{eqnarray}
\label{conditions for the second case} V(x)\sim\frac{\Psi(\theta)}{|x|^p},\quad\text{as }|x|\to\infty,
\end{eqnarray} where $\theta=x/|x|$ and $0<p<2$. The function $\Psi\ge 0$ is assumed to be continuous on the unit circle $\mathbb S=\{x\in\mathbb R^2:|x|=1\}$.

\begin{theorem}
\label{second theorem}
    Let $0<p<2$ and let $V$ satisfy \eqref{conditions for the second case}. Then
    \begin{eqnarray*}
    N(\lambda,\alpha)\sim\frac{\alpha^{2/p}}{4\pi}\int_{\mathbb R^2}[((\lambda+\Psi(\theta)|x|^{-p})_+^2-m^2)_+]^{1/2}\,dx,\quad\text{as }\alpha\to\infty.
\end{eqnarray*}
\end{theorem} Here, $a_+$ denotes the positive part of the number $a$, i.e.,
\begin{eqnarray*}
    a_+=\max\{a,0\}.
\end{eqnarray*}

Let us give a short summary of the past research in the area. Eigenvalues in gaps of the continuous spectrum were studied extensively for Schr\"odinger operators. For instance, R. Hempel proved in \cite{H} that the number of eigenvalues of the operator $-\Delta+f-\alpha V$ passing through the point $\lambda$ of a gap in the spectrum of the periodic operator $-\Delta+f$ obeys Weyl's law:
\begin{align}
    N(\lambda,\alpha)\sim(2\pi)^{-d}\omega_d\alpha^{d/2}\int_{\mathbb R^d}V^{d/2}\,dx,\qquad\text{as }\alpha\to\infty,
    \label{Weyl's law}
\end{align} where $\omega_d$ is volume of the unit ball in $\mathbb R^d$. Then M. Birman proved \eqref{Weyl's law} through a different approach in \cite{B}. Other relevant results could also be found in \cite{ADH,H2}.  Such results could be linked to the study of crystal color. When some ions in a crystal lattice are replaced by impurity ions, the resulting perturbation, described by the potential $\alpha V$, can create new energy levels. These energy levels are exactly the eigenvalues of $-\Delta +f -\alpha V$. They are responsible  for the absorption of the light of specific wavelengths, leading to the observed color of the crystal. Based on these remarks, the problem addressed in the present paper can be associated with the study of the color of graphene. 
 
Several mathematicians have studied eigenvalues in gaps of the spectrum with non-signdefinite perturbations (see \cite{DH} and \cite{S}). In this case, the eigenvalues of the operator are not monotone functions of $\alpha$, and the quantity $N(\lambda,\alpha)$ should be defined as the difference of the number of eigenvalues passing $\lambda$ in two opposite directions. Using this definition, Safronov \cite{S} generalized the formula \eqref{Weyl's law} to the case of non-signdefinite perturbations. The only change is that $V$ has to be replaced by $V_+$ on the right hand side of \eqref{Weyl's law}.

The graphene operator $D_m$ was introduced in the paper by Ferrulli, Laptev and Safronov  \cite{SLF}. However, the authors of \cite{SLF} analyzed complex eigenvalues of the perturbed operator $D(\alpha)$ with a non-self adjoint perturbation $V$ and a fixed value of $\alpha$. Our results differ from theorems in \cite{SLF} in a critical way: the perturbations that we consider are self-adjoint and $\alpha\to\infty$.

The current paper  overcomes  the main technical difficulty appearing  when one studies   operators having only   bounded  gaps in the spectrum: one cannot  use  Dirichlet-Neumann bracketing as  suggested in  
the  paper  by Alama, Deift, and Hempel \cite{ADH}.
While Dirichlet-Neumann bracketing is a powerful technique for analyzing the spectrum of semi-bounded  operators, 
it cannot be applied  when dealing with operators that are not semi-bounded.

\section{Preliminaries}

Here, we   provide necessary background information that is  needed  to understand the subsequent sections of the paper. 

Let ${\frak H}$ be a separable Hilbert  space.  The class of compact operators  on ${\frak H}$ will be denoted by $\frak S_\infty$.
For a compact operator $T\in {\frak S}_\infty$, the symbols $s_k(T)$ denote the singular values of
$T$ enumerated in the non-increasing order $(k\in \mathbb N)$ and counted in accordance with their multiplicity. Observe that $s^2_k(T)$ are eigenvalues of the operator $T^*T$. We set
\begin{eqnarray*}
    n(s,T)=\#\{k: s_k(T)> s\},\qquad s>0.
\end{eqnarray*} For a self-adjoint compact operator $T$, we also set
\begin{eqnarray*}
    n_\pm(s,T)=\#\{k : \pm\lambda_k(T)>s\},\qquad s>0,
\end{eqnarray*} where $\lambda_k(T)$ are eigenvalues of $T$. It follows that (see \cite{BS})
\begin{eqnarray*}
    n_\pm(s_1 + s_2, T_1 + T_2)\le n_\pm(s_1, T_1)+n_\pm(s_2, T_2), \qquad s_1,s_2>0.
\end{eqnarray*} 
A similar inequality holds for the function $n$. Also,
\begin{eqnarray*}
    n(s_1s_2, T_1T_2)\le n(s_1, T_1)+n(s_2,T_2),\qquad s_1, s_2>0.
\end{eqnarray*} The class of compact operators $T$ whose singular values satisfy
\begin{eqnarray*}
    \|T\|^p_{\frak S_p}:=\sum_ks^p_k(T)<\infty,\qquad p>0,
\end{eqnarray*} is called the Schatten class $\frak S_p$.

Besides the classes $\frak S_p$, we will be dealing with the so-called
weak Schatten classes $\Sigma_p$ of compact operators $T$ obeying the condition
\begin{eqnarray*}
    \|T\|^p_{\Sigma_p}:=\sup_{s>0}s^pn(s,T)<\infty.
\end{eqnarray*} 
We also introduce the class $\Sigma_p^0$ as the collection of compact operators $T$ such that
\begin{align*}
    n(s,T)=o(s^{-p}),\quad\text{as }s\to0.
\end{align*} Note that $\frak S_p\subset\Sigma^0_p$.

\begin{proposition}
\label{Holder's inequality}
    Let $T_1\in\Sigma_p$ and $T_2\in\Sigma_q$, where $p>0$ and $q>0$. Then $T_1T_2\in\Sigma_r$, where $1/r = 1/p + 1/q$, and 
    \begin{eqnarray*}
    \|T_1T_2\|_{\Sigma_r}\le 2^{1/r}\|T_1\|_{\Sigma_p}\|T_2\|_{\Sigma_q}.    \end{eqnarray*}
\end{proposition}
\begin{proof}
    Since $s=s^{r/p}s^{r/q}$, we have
    \begin{eqnarray*}
        n(s,T_1T_2)&\le& n(s^{r/p},T_1)+n(s^{r/q},T_2)\\
        &\le&s^r(\|T_1\|^p_{\Sigma_p}+\|T_2\|^q_{\Sigma_q}).
        \end{eqnarray*} Therefore,
        \begin{eqnarray}
        \label{sigma class inequality}
            \|T_1T_2\|^r_{\Sigma_r}\le\|T_1\|^p_{\Sigma_p}+\|T_2\|^q_{\Sigma_q}.       \end{eqnarray} Clearly, the inequality \eqref{sigma class inequality} implies the estimate
            \begin{eqnarray*}
                \left\|\frac{1}{\|T_1\|_{\Sigma_p}\|T_2\|_{\Sigma_q}}T_1T_2\right\|_{\Sigma_r}\le 2^{1/r},
            \end{eqnarray*} which can alternatively be written in the form $\|T_1T_2\|_{\Sigma_r}\le 2^{1/r}\|T_1\|_{\Sigma_p}\|T_2\|_{\Sigma_q}$.
\end{proof} 

For self-adjoint operators $T=T^*\in\Sigma_p$, we introduce the functionals
\begin{align*}
    \Delta^\pm_p(T):=\limsup_{s\to0}s^pn_\pm(s,T),\qquad\delta^\pm_p(T):=\liminf_{s\to0}s^pn_\pm(s,T)
\end{align*}

The following theorem is due to H. Weyl.
\begin{theorem}
\label{Weyl's theorem}
    If $T_1,T_2$ are self-adjoint operators, $T_1,T_2\in\Sigma_p$, and $T_1-T_2\in\Sigma_p^0$, then
    \begin{align*}
        \Delta^\pm_p(T_1)=\Delta^\pm_p(T_2),\qquad\delta^\pm_p(T_1)=\delta^\pm_p(T_2),
        \end{align*}
\end{theorem}

For $\lambda\in(-m,m)$, we define the operator $X_\lambda$ by
\begin{eqnarray}
\label{definition of Birman-Schwinger operator}
    X_\lambda=W(D_m-\lambda I)^{-1}W,\qquad W=\sqrt{V}.
\end{eqnarray} The next statement is widely  known as the Birman-Schwinger principle. 
It  allows one to transform the  spectral problem involving the  unbounded operator $D(\alpha)$
  into a spectral problem for the compact  operator $X_\lambda$.

\begin{proposition}
\label{B-S principle}
    Let $\alpha>0$ and $X_\lambda$ be defined by \eqref{definition of Birman-Schwinger operator}. Then
    \begin{eqnarray*}
        N(\lambda,\alpha)=n_+(s,X_\lambda),\qquad\text{for }s=\alpha^{-1}.
    \end{eqnarray*}
\end{proposition}

For the proof of this statement, see \cite{B}.\\

We will also need the following theorem (see M. Birman \cite{B}) in which $-\Delta$ is the operator on ${\Bbb R}^d$ with an arbitrary $d\in {\Bbb N}$.

\begin{theorem}
\label{three cases} Let $W$ be a real-valued   function on  ${\Bbb R}^d$.
Let $X=W((-\Delta)^l+I)^{-1}W$, $\,V=W^2$ and $p=d/2l$. Assume that $[V]_p<\infty$, where
\begin{equation*}
[V]_p=\begin{cases}
    \sum_{n\in\mathbb Z^d}\|V\|_{L^1(\mathbb Q_n)}^p,&\text{ if }0<p<1;\\
    \sum_{n\in\mathbb Z^d}(\int_{\mathbb Q_n} V^q\,dx)^{1/q},\forall q>1,&\text{ if }p=1;\\\|V\|_p, &\text{ if }p>1,
\end{cases} 
\end{equation*} and $\mathbb Q_n=[0,1)^d+n$ for $n\in\mathbb Z^d$. Then $X\in\Sigma_p$ and $\|X\|_{\Sigma_p}\le C[V]_p$ for some constant $C>0$ independent of $V$. 
\end{theorem}

With $[V]_p$ defined as in Theorem ~\ref{three cases}, we formulate the following corollary.

\begin{corollary}
\label{corollary of three cases}
    Assume that $X=W((-\Delta)^l+I)^{-1}\tilde W,V=W^2,\tilde V=\tilde W^2$ and $p=d/2l$. Then  $\|X\|_{\Sigma_p}\le C_0[V]_p^{1/2}[\tilde V]_p^{1/2}$ for some constant $C_0>0$ independent of $V$ and $\tilde V$.
\end{corollary}

\section{Proof of Theorem ~\ref{first theorem}}

Let $F$ be the Fourier transform on $L^2(\mathbb R^2)$ defined by
\begin{eqnarray*}
    [Fu](\xi)=\frac{1}{2\pi}\int_{\mathbb R^d}e^{-i\xi x}u(x)\,dx.
\end{eqnarray*} Then $F$ diagonalizes the graphene operator in the sense that
\begin{eqnarray*}
    D_m=\begin{pmatrix}
        m&(\frac{\partial}{\partial x_1}+\frac{1}{i}\frac{\partial}{\partial x_2})^2\\(\frac{\partial}{\partial x_1}-\frac{1}{i}\frac{\partial}{\partial x_2})^2&-m    \end{pmatrix}=F^*\begin{pmatrix}
            m&(i\xi_1+\xi_2)^2\\(i\xi_1-\xi_2)^2&-m
        \end{pmatrix}F.
\end{eqnarray*} For convenience, we denote
\begin{eqnarray*}
    \hat{D}_m(\xi)=\begin{pmatrix}
            m&(i\xi_1+\xi_2)^2\\(i\xi_1-\xi_2)^2&-m
        \end{pmatrix}
        =\begin{pmatrix}
            m&-(\xi_1-i\xi_2)^2\\
            -(\xi_1+i\xi_2)^2&-m
        \end{pmatrix},
\end{eqnarray*} and call it the symbol of $D_m$. Observe that
\begin{eqnarray}
\label{matrix-valued function}
    (D_m-\lambda I)^{-1}&=&F^*[(\hat{D}_m(\xi)-\lambda I)^{-1}]F
    =\begin{pmatrix}
        m-\lambda&-(\xi_1-i\xi_2)^2\\-(\xi_1+i\xi_2)^2&-m-\lambda
    \end{pmatrix}^{-1}\notag\\&
    =&\frac{1}{\lambda^2-m^2-|\xi|^4}\begin{pmatrix}
        -m-\lambda&(\xi_1-i\xi_2)^2\\(\xi_1+i\xi_2)^2&m-\lambda
    \end{pmatrix}.
\end{eqnarray} We now show that it is enough to prove Theorem ~\ref{first theorem} for the case where $W\in C_0^\infty(\mathbb R^2)$. 

Let $W\not\in C_0^\infty(\mathbb R^2)$ satisfy the conditions of Theorem ~\ref{first theorem}. For an arbitrary $\varepsilon>0$, let $W_\varepsilon$ be a smooth compactly supported approximation of $W$ with the property
\begin{align*}
    [(W-W_\varepsilon)^2]_1<\varepsilon.
\end{align*} Then apparently, the operator $X_\varepsilon(\lambda)=W_\varepsilon(D_m-\lambda I)^{-1}W_\varepsilon$ approximates $X(\lambda)=W(D_m-\lambda I)^{-1}W$ in the class $\Sigma_1$.

In order to illustrate that, we apply Corollary ~\ref{corollary of three cases}, which leads  us to the following statement.

\begin{proposition}
\label{proposition 3.1}
Let $W$ and $\tilde W$ be two real-valued functions on $\mathbb R^2$ such that 
\begin{align*}
    [W^2]_1<\infty\quad\text{and}\quad[\tilde W^2]_1<\infty,
\end{align*}
   Then 
\begin{align*}
    \|W(D_m-\lambda I)^{-1}\tilde W\|_{\Sigma_1}\le C[W]^{1/2}_1[\tilde W]^{1/2}_1,
\end{align*} for some constant $C\ge 0$ that is independent of  both $W$ and $\tilde W$.
\end{proposition} 

\begin{proof}
    Indeed, because
    \begin{align*}
        (D_m-\lambda I)^{-1}=(-\Delta+I)^{-1/2}(-\Delta+I)^{1/2}(D_m-\lambda I)^{-1}(-\Delta+I)^{1/2}(-\Delta+I)^{-1/2},
    \end{align*} it suffices to show that the operator
    \begin{align*}
        B(\lambda)=(-\Delta+I)^{1/2}(D_m-\lambda I)^{-1}(-\Delta+I)^{1/2}
    \end{align*} initially defined on $C_0^\infty(\mathbb R^2)$ has a bounded extension to all of $L^2(\mathbb R^2)$. The boundedness of $B(\lambda)$ is equivalent to the assertion that the matrix-valued function
    \begin{align*}
        (|\xi|^2+1)^{1/2}(\hat D_m(\xi)-\lambda I)^{-1}(|\xi|^2+1)^{1/2}
    \end{align*} is bounded. The latter follows straightforwardly from \eqref{matrix-valued function}.
\end{proof}

As a direct consequence of Proposition ~\ref{proposition 3.1}, we obtain
\begin{align*}
    \|X(\lambda)-X_\varepsilon(\lambda)\|&\le\|(W-W_\varepsilon)(D_m-\lambda I)^{-1}W\|_{\Sigma_1}+\|W_\varepsilon(D_m-\lambda I)^{-1}(W-W_\varepsilon)\|_{\Sigma_1}\\
    &\le 2C([(W-W_\varepsilon)^2]_1^{1/2}[W^2]_1^{1/2}+[W_\varepsilon^2]_1^{1/2}[(W-W_\varepsilon)^2]^{1/2})\\
    &<2C\sqrt{\varepsilon}([W^2]_1^{1/2}+[W_\varepsilon^2]_1^{1/2}])=:\delta(\varepsilon),
\end{align*} which tends to $0$ when $\varepsilon\to 0$. Moreover, $X(\lambda)=X_\varepsilon(\lambda)+(X(\lambda)-X_\varepsilon(\lambda))$. By Ky Fan inequality, it follows that
\begin{align}
\label{delta}
    n_+(s,X(\lambda))&\le n_+((1-\varepsilon_0)s,X_\varepsilon(\lambda))+n_+(\varepsilon_0 s,X(\lambda)-X_\varepsilon(\lambda))\notag\\
    &\le n_+((1-\varepsilon_0)s,X_\varepsilon(\lambda))+(\varepsilon_0s)^{-1}\|X(\lambda)-X_\varepsilon(\lambda)\|_{\Sigma_1}\notag\\
    &\le n_+((1-\varepsilon_0)s,X_\varepsilon(\lambda))+(\varepsilon_0s)^{-1}\delta(\varepsilon),
\end{align} for any $0<\varepsilon_0<1$. Suppose that Theorem ~\ref{first theorem} holds for the potential $V_\varepsilon=W_\varepsilon^2$, that is,
\begin{align*}
\lim_{s\to0} sn_+(s,X_\varepsilon(\lambda))=\frac{1}{4\pi}\int_{\mathbb R^2}V_\varepsilon\,dx.
\end{align*} Then we deduce from \eqref{delta} that
\begin{align*}
    &\limsup_{s\to0}sn_+(s,X(\lambda))\le\limsup_{s\to0}sn_+((1-\varepsilon_0)s,X_\varepsilon(\lambda))+\delta(\varepsilon)/\varepsilon_0\\
    &=\frac{1}{1-\varepsilon_0}\limsup_{s\to0}sn_+(s,X_\varepsilon(\lambda))+\delta(\varepsilon)/\varepsilon_0
    =\frac{1}{4\pi(1-\varepsilon_0)}\int_{\mathbb R^2} V_\varepsilon\,dx+\delta(\varepsilon)/\varepsilon_0.
\end{align*} By taking the limit as $\varepsilon\to0$, we infer that
\begin{align*}
    \limsup_{s\to0}sn_+(s,X(\lambda))\le \frac{1}{4\pi(1-\varepsilon_0)}\int_{\mathbb R^2}V\,dx.
\end{align*} Due to the fact that $\varepsilon_0>0$ is arbitrary, we obtain
\begin{align*}
    \limsup_{s\to0}sn_+(s,X(\lambda))\le \frac{1}{4\pi}\int_{\mathbb R^2}V\,dx.
\end{align*} Similarly, we can establish the inequality
\begin{align*}
    \liminf_{s\to0}sn_+(s,X(\lambda))\ge \frac{1}{4\pi}\int_{\mathbb R^2}V\,dx.
\end{align*} Thus, Theorem ~\ref{first theorem} holds  for  potentials $V\notin C_0^\infty$ as long as it holds  for all  $V=W^2$ with $W\in C_0^\infty$.

Now we assume that $W\in C_0^\infty(\mathbb R^2)$, and choose an appropriate function $\zeta\in C^\infty(\mathbb R^2)$ with the property
\begin{eqnarray*}
    \zeta(\xi)=\begin{cases}
        0,\quad\text{ if } |\xi|<\varepsilon;\\
        1,\quad\text{ if $|\xi|>2$},
    \end{cases}
\end{eqnarray*} where $\varepsilon>0$ is sufficiently small. Then we define the matrix-valued function
\begin{eqnarray*}
    a(\xi)=\frac{\zeta(\xi)}{|y|^4}\begin{pmatrix}
        0&-(\xi_1-i\xi_2)^2\\-(\xi_1+i\xi^2)&0
    \end{pmatrix}.
\end{eqnarray*} Additionally, we set 
\begin{align*}
\Tilde{X}(\lambda)=\sqrt{V}F^*[a(\xi)]F\sqrt{V},
\end{align*} where $[a(\xi)]$ denotes the operator of multiplication by the function $a(\xi)$, and prove  a valuable proposition  below.

\begin{proposition} Let $W\in C_0^\infty(\mathbb R^2)$. Suppose that $\lim_{s\to0}sn_+(s,\tilde X(\lambda))$ exists. Then
\begin{align*}
    \lim_{s\to0}sn_+(s,\tilde X(\lambda))=\lim_{s\to0}sn_+(s,X(\lambda)).
\end{align*}
\end{proposition} 

\begin{proof} It is adequate to prove that $\tilde X(\lambda)-X(\lambda)\in\Sigma_1^0$. Observe that 
\begin{eqnarray*}
    \tilde X(\lambda)-X(\lambda)=WF^*[\beta(\xi)]FW,
\end{eqnarray*}
where
    \begin{eqnarray}
        \beta(\xi)&=&(\hat{D}_m(\xi)-\lambda I)^{-1}-\frac{\zeta(\xi)}{|\xi|^4}\begin{pmatrix}
            0&-(\xi_1-i\xi_2)^2\\-(\xi_1+i\xi_2)^2&0
        \end{pmatrix}\nonumber\\
        &=&\frac{1}{m^2-\lambda^2+|\xi|^4}\begin{pmatrix}
            m+\lambda&0\\0&\lambda-m
        \end{pmatrix}\nonumber\\&&+
        \left(\frac{1}{m^2-\lambda^2+|\xi|^4}-\frac{\zeta(\xi)}{|\xi|^4}\right)\begin{pmatrix}
            0&-(\xi_1-i\xi_2)^2\\-(\xi_1+i\xi_2)^2&0
        \end{pmatrix}.\label{difference}
    \end{eqnarray} The first term on the right-hand side of \eqref{difference} is an integrable function of $\xi$. And regarding the second term, we have
    \begin{eqnarray*}
        \frac{1}{m^2-\lambda^2+|\xi|^4}-\frac{\zeta(\xi)}{|y|^4}=\frac{(1-\zeta(\xi))|\xi|^4-\zeta(\xi)(m^2-\lambda^2)}{|\xi|^4(m^2-\lambda^2+|\xi|^4)},
    \end{eqnarray*} which is evidently integrable in some neighborhood of the origin by the definition of $\zeta$ and equals $O(|\xi|^{-8})$ as $|\xi|\to\infty$. Thus, $\beta\in L^1(\mathbb R^2)$. This implies that the operators $WF^*|\beta(\xi)|^{1/2}$ and $|\beta(\xi)|^{1/2}FW$ are both Hilbert-Schmidt operators. Therefore,
    \begin{align*}
        \tilde X(\lambda)-X(\lambda)=WF^*\big[|\beta(\xi)|^{1/2}\text{sign}(\beta(\xi))|\beta(\xi)|^{1/2}\big]FW
    \end{align*} belongs to $\frak S_1\subset\Sigma_1^0$.
\end{proof}

The following theorem is a consequence of Theorem 2 from \cite{BS}.
\begin{theorem} Let $W\in C_0^\infty(\mathbb R^2)$. Then
    \begin{eqnarray}
        \lim_{s\to0}(sn_+(s,X))=\frac{1}{(2\pi)^2}\int_{\mathbb R^2} \int_{\mathbb R^2}n_+(1,G(x,\xi))\,dx\,d\xi,\label{double integral}
    \end{eqnarray} where
    \begin{eqnarray}
    \label{G}    G(x,\xi)=V(x)|\xi|^{-4}\begin{pmatrix}
            0&-(\xi_1-i\xi_2)^2\\-(\xi_1+i\xi_2)^2&0
        \end{pmatrix}.
    \end{eqnarray}
\end{theorem} 

Hence, for every fixed $\xi\in\mathbb R^2$, the trace of operator $G(x,\xi)$ in \eqref{G} is
\begin{eqnarray*}
    \text{tr}(G(x,\xi))=0,
\end{eqnarray*} which implies that the sum of the eigenvalues of $G(x,\xi)$ is zero: $\lambda_-+\lambda_+=0$. Besides computing the trace, we  calculate the  determinant
\begin{eqnarray*}
    \det(G(x,\xi))=\lambda_+\cdot\lambda_-=-\frac{V^2(x)}{|\xi|^8}\cdot|\xi|^4=-V^2(x)|\xi|^{-4}.
\end{eqnarray*} Thus, $\lambda_+=V(x)|\xi|^{-2}, \lambda_-=-V(x)|\xi|^{-2}$, which yields that $n_+(1,G(x,\xi))\le 1$. More precisely,
\begin{eqnarray*}
    n_+(1,G(x,\xi))=\begin{cases}
        0, \quad\text{if } V(x)|\xi|^{-2}\le 1;\\
        1,\quad\text{if } V(x)|\xi|^{-2}>1.
    \end{cases}
\end{eqnarray*} Consequently, $n_+(1,G(x,\xi))$ is the characteristic function $\chi_\Omega$ of the set
\begin{eqnarray*}
    \Omega:=\{(x,\xi)\in\mathbb R^2\times\mathbb R^2:V(x)|\xi|^{-2}>1\}.
\end{eqnarray*} which implies
\begin{align}
    &\frac{1}{(2\pi)^2}\int_{\mathbb R^2} \int_{\mathbb R^2}n_+(1,G(x,\xi))\,dx\,d\xi=\frac{1}{(2\pi)^2}\int_{\mathbb R^2}\left(\int_{\mathbb R^2}\chi_\Omega(x,\xi)\,d\xi\right)\,dx\notag
    \\&=\frac{1}{(2\pi)^2}\int_{\mathbb R^2}\pi V(x)\,dx=\frac{1}{4\pi}\int_{\mathbb R^2}V(x)\,dx,\label{second double integral}
\end{align} Combining \eqref{double integral} and \eqref{second double integral} with Proposition ~\ref{B-S principle}, we establish the asymptotic relation
\begin{eqnarray*}
    N(\lambda,\alpha)\sim\frac{\alpha}{4\pi}\int_{\mathbb R^2}V(x)\,dx,\quad\text{as }\alpha\to\infty,
\end{eqnarray*} for every $V$ that obeys \eqref{conditions for the first case}. The proof of Theorem ~\ref{first theorem} is complete. $\,\,\,\blacksquare$

\section{Proof of Theorem ~\ref{second theorem}}

For $0<\varepsilon_1<\varepsilon_2<\infty$, we decompose $\mathbb R^2$ into the union of the following three subsets:
\begin{eqnarray*}
\Omega_1(\alpha)&=&\{x\in {\mathbb R}^2: \,\, |x|<\varepsilon_1\alpha^{1/p}\},
\\
\Omega_2(\alpha)&=&\{x\in\mathbb R^2:\varepsilon_1\alpha^{1/p}\le|x|\le\varepsilon_2\alpha^{1/p}\},
\\
\Omega_3(\alpha)&=&\{x\in {\mathbb R}^2: \,\, |x|>\varepsilon_2\alpha^{1/p}\},
\end{eqnarray*} and set $\chi_1,\chi_2$ and $\chi_3$ to be the characteristic function of $\Omega_1(\alpha), \Omega_2(\alpha)$ and $\Omega_3(\alpha)$, respectively. Also, we define the three operators $W_i=\chi_i W$ for $i=1,2,3$. Clearly, $W=\sum_{i=1}^3W_i$, which leads to
\begin{eqnarray}
\label{decomposition}
    X_\lambda&=&W(D_m-\lambda I)^{-1}W=\sum_{i=1}^3\sum_{j=1}^3W_i(D_m-\lambda I)^{-1}W_j.
\end{eqnarray} Only the three operators of form $W_i(D_m-\lambda I)^{-1}W_i$ in \eqref{decomposition} contribute to the asymptotics of $n(s,X_\lambda)$. In some sense, the sum on the right-hand side of \eqref{decomposition} can be replaced by the operator
\begin{eqnarray*}
    X_\lambda^+=W_1(D_m-\lambda I)^{-1}W_1+W_2(D_m-\lambda I)^{-1}W_2+W_3(D_m-\lambda I)^{-1}W_3.    
\end{eqnarray*} We will show that the contribution of the remaining six operators in \eqref{decomposition} to the asymptotics of $n(s,X_\lambda)$ is negligible. Namely, for $i\ne j$, we will prove that 
\begin{eqnarray}
\label{asymptotic behavior of term containing Wi and Wj}    n(\varepsilon\alpha^{-1},W_i(D_m-\lambda I)^{-1}W_j)=o(\alpha^{2/p}), \qquad\text{as }\alpha\to\infty,\quad\forall\varepsilon>0.
\end{eqnarray} By Ky Fan Inequality, \eqref{asymptotic behavior of term containing Wi and Wj} would indicate that 
\begin{eqnarray}
\label{result of KF inequality}
    n(\varepsilon\alpha^{-1},X_\lambda^-)=o(\alpha^{2/p}),\qquad\text{as }\alpha\to\infty,
\end{eqnarray} where
\begin{eqnarray}
\label{expression for X-}
    X_\lambda^-=\sum_{i\ne j}W_i(D_m-\lambda I)^{-1}W_j.
\end{eqnarray}

\begin{proposition}
\label{integral}
    Let \eqref{result of KF inequality} hold for any $\varepsilon>0$. Assume that for every $\tau>0$, $n_+(\tau\alpha^{-1},X_\lambda^+)\sim\tau^{-2/p}\alpha^{2/p}J(\lambda,m)$ as $\alpha\to\infty$, where
    \begin{eqnarray*}
        J(\lambda,m)=\frac{1}{4\pi}\int_{\mathbb R^2}[((\lambda+\Psi(\theta)|x|^{-p})_+^2-m^2)_+]^{1/2}\,dx.    \end{eqnarray*} Then
        \begin{eqnarray*}
            n(\alpha^{-1},X_\lambda)\sim\alpha^{2/p}J(\lambda,m),\qquad\text{as } \alpha\to\infty.
        \end{eqnarray*}
\end{proposition}

\begin{proof}
    Indeed, for any $\varepsilon>0$,
    \begin{eqnarray*}
        n(\alpha^{-1},X_\lambda)\le n_+((1-\varepsilon)\alpha^{-1},X_\lambda^+)+n(\varepsilon\alpha^{-1},X_\lambda^-),
    \end{eqnarray*} and
    \begin{eqnarray*}
        n(\alpha^{-1},X_\lambda)\ge n_+(((1-\varepsilon)\alpha)^{-1},X_\lambda^+)-n(\frac{\varepsilon}{1-\varepsilon}\alpha^{-1},X_\lambda^-).
    \end{eqnarray*} Therefore,
    \begin{eqnarray*}
    (1-\varepsilon)^{2/p}J(\lambda,m)\le\liminf_{\alpha\to\infty}\alpha^{-2/p}n_+(\alpha^{-1},X_\lambda)\le\limsup_{\alpha\to\infty}\alpha^{-2/p}n_+(\alpha^{-1},X_\lambda)\leq\frac{J(\lambda,m)}{(1-\varepsilon)^{2/p}}.
    \end{eqnarray*} Taking the limit as $\varepsilon\to0$, we obtain the objective relation.
\end{proof}

Let us first prove \eqref{result of KF inequality}. Evidently, operator $W_i(D_m-\lambda I)^{-1}W_j$ is self-adjoint if $i=j$ and non-self-adjoint if $i\ne j$. Also, $W_i(D_m-\lambda I)^{-1}W_j$ is the adjoint of $W_j(D_m-\lambda I)^{-1}W_i$. Hence, $X_\lambda^+$ and $X_\lambda^-$ are both self-adjoint operators. Since the singular values of the operators $W_i(D_m-\lambda I)^{-1}W_j$ and $W_i(D_m-\lambda I)^{-1}W_j$ are the same, we need to consider only three operators in \eqref{expression for X-}.

\subsection{Operator $W_1(D_m-\lambda I)^{-1}W_2$}

We intend to demonstrate that
\begin{eqnarray}
\label{W1W2}
    n(\varepsilon\alpha^{-1},W_1(D_m-\lambda I)^{-1}W_2)=o(\alpha^{2/p}),\qquad\text{as }\alpha\to\infty.
\end{eqnarray}
We fix $\delta\in(0,\infty)$ and select a function $\zeta\in C^\infty(\mathbb R)$ satisfying the condition
\begin{eqnarray}
\label{zeta}
    \zeta(t)=\begin{cases}
        0,\text{ if }t<0;\\
        1,\text{ if }t>\delta.
    \end{cases}
\end{eqnarray} Then we define another function $\theta\in C^\infty(\mathbb R^2)$ by setting $\theta(x)=\zeta(|x|-\varepsilon_1\alpha^{1/p})$. Observe that
\begin{eqnarray}
    W_1(D_m-\lambda I)^{-1}W_2&=W_1(D_m-\lambda I)^{-1}\theta W_2+ W_1(D_m-\lambda I)^{-1}(1-\theta) W_2\nonumber\\
    =&W_1[(D_m-\lambda I)^{-1},\theta]\,W_2+ W_1(D_m-\lambda I)^{-1}(1-\theta) W_2\label{another form of the first operator},
\end{eqnarray} where $[A,B]=AB-BA$. Also, it is easy to verify that $[A^{-1},B]=A^{-1}(BA-AB)A^{-1}$, provided that $A$ is invertible. Therefore,
\begin{eqnarray}
\label{results for the algebraic relation}
    [(D_m-\lambda I)^{-1},\theta]&=&(D_m-\lambda I)^{-1}(\theta(D_m-\lambda I)-(D_m-\lambda I)\theta)(D_m-\lambda I)^{-1}\nonumber\\
    &=&(D_m-\lambda I)^{-1}(\theta D_m-D_m\theta)(D_m-\lambda I)^{-1}\label{commutator result}.
\end{eqnarray} A brief computation yields that
\begin{eqnarray}
\label{commutator}
\nonumber
    \theta D_m-D_m\theta=\theta D_0-D_0\theta=Y_0+Y_1,
\end{eqnarray} where
\begin{eqnarray*}
    Y_0&=&\begin{pmatrix}
        0&\theta_1\\\theta_2&0
    \end{pmatrix},
    Y_1=2\begin{pmatrix}
        0&\theta_3(\frac{\partial}{\partial x_1}-i\frac{\partial}{\partial x_2})\\\theta_4(\frac{\partial}{\partial x_1}+i\frac{\partial}{\partial x_2})&0
    \end{pmatrix},\\
    \theta_1&=&\frac{\partial^2\theta}{\partial x_1^2}-2i\frac{\partial\theta}{\partial x_1\partial x_2}-\frac{\partial^2\theta}{\partial x_2^2},
    \theta_2=\frac{\partial^2\theta}{\partial x_1^2}+2i\frac{\partial\theta}{\partial x_1\partial x_2}-\frac{\partial^2\theta}{\partial x_2^2},\\
    \theta_3&=&\frac{\partial\theta}{\partial x_1}-i\frac{\partial\theta}{\partial x_2},\quad
    \theta_4=\frac{\partial\theta}{\partial x_1}+i\frac{\partial\theta}{\partial x_2}.    
\end{eqnarray*} Multiplying  the equality \eqref{results for the algebraic relation} by $W_1$ and $W_2$,  we obtain
\begin{eqnarray}
\label{Y0 and Y1}
    W_1((D_m-\lambda I)^{-1}\theta-\theta(D_m-\lambda I)^{-1})W_2&=&W_1(D_m-\lambda I)^{-1}(Y_0+Y_1)(D_m-\lambda I)^{-1}W_2\nonumber\\
    &=&X_{1,2}+\tilde X_{1,2},
\end{eqnarray} where
\begin{eqnarray*}
    X_{1,2}&=&W_1(D_m-\lambda I)^{-1}Y_0(D_m-\lambda I)^{-1}W_2,\\
    \tilde X_{1,2}&=&W_1(D_m-\lambda I)^{-1}Y_1(D_m-\lambda I)^{-1}W_2.
\end{eqnarray*} Let us also denote $\hat{X}_{1,2}=W_1(D_m-\lambda I)^{-1}(1-\theta)W_2$. To prove \eqref{asymptotic behavior of term containing Wi and Wj} for $i=1,j=2$, it is sufficient to establish that $n(\varepsilon\alpha^{-1},X_{1,2})=o(\alpha^{2/p}),n(\varepsilon\alpha^{-1},\tilde X_{1,2})=o(\alpha^{2/p})$ and $n(\varepsilon\alpha^{-1},\hat X_{1,2})=o(\alpha^{2/p})$ as $\alpha\to\infty$.

Let us examine the operator $X_{1,2}$. Let $\chi_\delta$ denote the characteristic function of the layer $\{x\in\mathbb R^2: \varepsilon_1\alpha^{1/p}<|x|<\varepsilon_1\alpha^{1/p}+\delta\}$. Since the supports of the functions $\theta_1$ and $\theta_2$ are both contained in the layer, one may represent operator $X_{1,2}$ as
\begin{eqnarray*}
\label{reformulation of X12}
    X_{1,2}=W_1(D_0-\lambda I)^{-1}\chi_\delta Y_0\chi_\delta(D_0-\lambda I)^{-1}W_2.
\end{eqnarray*} In addition, by the definition of $Y_0$, we have $\|Y_0\|\le\|\theta_1\|_\infty+\|\theta_2\|_\infty$. Based on H\"older's inequality for the $\Sigma_p$ class stated in Proposition ~\ref{Holder's inequality}, we infer that
\begin{eqnarray*}
    \|X_{1,2}\|_{\Sigma_{1/2}}\le4(\|\theta_1\|_\infty+\|\theta_2\|_\infty)\|W_1(D_m-\lambda I)^{-1}\chi_\delta\|_{\Sigma_1}\cdot\|\chi_\delta(D_m-\lambda I)^{-1}W_2\|_{\Sigma_1}.
\end{eqnarray*} Moreover,
\begin{eqnarray}
\label{very important result}
    F(D_m-\lambda I)^{-1}F^*=\frac{1}{m^2-\lambda^2+|\xi|^4}\begin{pmatrix}
        \lambda+m&(\xi_1-i\xi_2)^2\\(\xi_1+i\xi_2)^2&\lambda-m
    \end{pmatrix}
\end{eqnarray} which tends to zero as $O(|\xi|^{-2})$ as $|\xi|\to\infty$. Therefore, there exists some function $b\in L^\infty(\mathbb R^2)$ which fulfills the relation
\begin{eqnarray*}
    F(D_m-\lambda I)^{-1}F^*= \frac{b(\xi)}{1+|\xi|^2}.
\end{eqnarray*} Consequently, we obtain
\begin{eqnarray}
    &\|W_1(D_m-\lambda I)^{-1}\chi_\delta\|_{\Sigma_1}\le\|W_1F^*(|\xi|^2+1)^{-1/2}b(\xi)(|\xi|^2+1)^{-1/2}F\chi_\delta\|_{\Sigma_1}\nonumber\\
    &\le\|b\|_\infty\|W_1F^*(|\xi|^2+1)^{-1/2}\|_{\Sigma_2}\|(|\xi|^2+1)^{-1/2}F\chi_\delta\|_{\Sigma_2}\label{initial estimate}.
\end{eqnarray} Through Corollary ~\ref{corollary of three cases}, we conclude that
\begin{align}
    &\|W_1F^*(|\xi|^2+1)^{-1/2}\|^2_{\Sigma_2}\nonumber
    =\|W_1F^*(|\xi|^2+1)^{-1}FW_1\|_{\Sigma_1}\nonumber\\
    &=\|W_1((-\Delta)+I)^{-1}W_1\|_{\Sigma_1}\label{one estimate}\le C_0[W_1^2]_1.
\end{align} Likewise, we have
\begin{eqnarray}
\label{another estimate}
    \|(|\xi|^2+1)^{-1/2}F\chi_\delta\|^2_{\Sigma_2}\le C_0[\chi_\delta]_1.
\end{eqnarray} 

Combining \eqref{initial estimate}, \eqref{one estimate} and \eqref{another estimate}, we come to  the following estimate:
\begin{eqnarray*}
    \|W_1(D_m-\lambda I)^{-1}\chi_\delta\|_{\Sigma_1}&\le& C_0\|b\|_\infty[\chi_\delta]^{1/2}_1\cdot[W^2_1]^{1/2}_1\\
    &=&C_1\left(\sum_{n\in\mathbb Z^d}\|W_1^2\|_{L^q(\mathbb Q_n)}\right)^{1/2}[\chi_\delta]_1^{1/2}\\
    &\le&C_1\left(\sum_{n\in\mathbb Z^d,|n|<\varepsilon_1\alpha^{1/p}}\frac{1}{|n|^p+1}\right)^{1/2}(O(\alpha^{1/p}))^{1/2}\\
    &\le&C_2\left(\int_{|x|<\varepsilon_1\alpha^{1/p}}\frac{dx}{|x|^p+1}\right)^{1/2}O(\alpha^{1/2p})\\
    &\asymp& C_2\left(\int_0^{2\pi}\int_{r<\varepsilon_1\alpha^{1/p}}r^{1-p}\,dr\,d\theta\right)^{1/2}O(\alpha^{1/2p})\\
    &=&\left(2\pi  C_2\left[\frac{r^{2-p}}{2-p}\right]\Big|_{r=0}^{r=\varepsilon_1\alpha^{1/p}}\right)^{1/2}O(\alpha^{1/2p})\\
    &=&(C_3\varepsilon_1^{2-p}\alpha^{\frac{2}{p}-1})^{1/2}O(\alpha^{1/2p})=C_4\varepsilon_1^{1-\frac{p}{2}}O(\alpha^{\frac{3}{2p}-\frac{1}{2}}),
\end{eqnarray*} where $C_1,C_2,C_3$ and $C_4$ are non-negative constants. In the estimate above, we used the fact that
\begin{eqnarray*}
    [\chi_\delta]_1 &=& \sum_{n \in \mathbb{Z}^d} \|\chi_\delta\|_{L^q(\mathbb{Q}_n)}\asymp \sum_{\varepsilon_1 \alpha^{1/p} < |n| < \varepsilon_1 \alpha^{1/p} + \delta} 1 \\
    &=& \text{Area} \{ x \in \mathbb{R}^2 : \varepsilon_1 \alpha^{1/p} < |x| < \varepsilon_1 \alpha^{1/p} + \delta \} \\
    &=& \pi \left[ (\varepsilon_1 \alpha^{1/p} + \delta)^2 - (\varepsilon_1 \alpha^{1/p})^2 \right]=\pi \left( 2 \varepsilon_1 \alpha^{1/p} \delta + \delta^2 \right).
\end{eqnarray*} Therefore,
\begin{eqnarray*}
    \|X_{1,2}\|_{\Sigma_{1/2}}&\le&4(\|\theta_1\|_\infty+\|\theta_2\|_\infty)\|W_1(D_m-\lambda I)^{-1}\chi_\delta\|_{\Sigma_1}\cdot\|\chi_\delta(D_m-\lambda I)^{-1}W_2\|_{\Sigma_1}\\
    &=&O(\alpha^{\frac{3}{p}-1}),\qquad\text{as } \alpha\to\infty.
\end{eqnarray*} Due to
\begin{eqnarray*}
    \sqrt{s}n(s,X_{1,2})\le\|X_{1,2}\|_{\Sigma_{1/2}}^{1/2},
\end{eqnarray*} we draw the conclusion that
\begin{eqnarray*}
    n(\varepsilon\alpha^{-1},X_{1,2})\le O(\alpha^{\frac{3}{2p}})=o(\alpha^{2/p}),
\end{eqnarray*} as $\alpha\to\infty$. The study of operator $X_{1,2}$ is complete.\\

Let us deal with  the operator $\tilde X_{1,2}$. Note that the order of the differential operator $Y_1$ is 1 and that of the Laplace operator $\Delta$ is 2, which implies that the operator $Y_1(-\Delta+I)^{-1/2}$ is bounded. Indeed, through direct computation, we derive the following inequality:
\begin{eqnarray}
\label{upper estimate of Y1u}
\|Y_1u\|^2&=&\int_{\mathbb R^2}\Big|\theta_3\left(\frac{\partial u_2}{\partial x_1}-i\frac{\partial u_2}{\partial x_2}\right)\Big|^2\,dx+\int_{\mathbb R^2}\Big|\theta_4\left(\frac{\partial u_1}{\partial x_1}+i\frac{\partial u_1}{\partial x_2}\right)\Big|^2\,dx\nonumber\\&\le&8\left(\|\theta_3\|^2_\infty\int_{\mathbb R^2}|\nabla u_2|^2\,dx+\|\theta_4\|^2_\infty\int_{\mathbb R^2}|\nabla u_1|^2\,dx\right).
\end{eqnarray} On the other hand, we have
\begin{eqnarray}
\label{estimate concerning gredient}
    \|(-\Delta+I)^{1/2}u\|^2=\int_{\mathbb R^2}(|\nabla u|^2+|u|^2)\,dx,
\end{eqnarray} where
\begin{eqnarray*}
    \nabla u=\left(
        \frac{\partial u}{\partial x_1},
        \frac{\partial u}{\partial x_2}
    \right)^T.
\end{eqnarray*} Incorporating \eqref{upper estimate of Y1u} and \eqref{estimate concerning gredient}, we obtain
\begin{eqnarray}
    \|Y_1u\|^2&\le&8(\|\theta_3\|_\infty^2+\|\theta_4\|_\infty^2)\cdot\left(\int_{\mathbb R^2}\left(|\nabla u_1|^2+|\nabla u_2|^2\right)\,dx\right)\nonumber\\
    &\le&C_\delta\left(\int_{\mathbb R^2}\left(|\nabla u_1|^2+|u_1|^2+|\nabla u_2|^2+|u_2|^2\right)\,dx\right)\nonumber\\
    &=&C_\delta\|(-\Delta+I)^{1/2}u\|^2_{L^2},\qquad\forall u\in L^2,\label{estimate of Y1u}
\end{eqnarray} where $C_\delta=8(\|\theta_3\|_\infty^2+\|\theta_4\|_\infty^2)$ is a constant depending only on $\delta$. Making the substitution $u=(-\Delta+I)^{-1/2}v$ in \eqref{estimate of Y1u}, we find that
\begin{eqnarray*}
    \|Y_1(-\Delta+I)^{-1/2}v\|\le C_\delta\|v\|,\qquad\forall v\in L^2,
\end{eqnarray*} which means that $Y_1(-\Delta+I)^{-1/2}$ is a bounded operator.

One can also see that the operator $\tilde X_{1,2}$ may be written as
\begin{eqnarray*}
    \tilde X_{1,2}=W_1(D_m-\lambda I)^{-1}\chi_\delta[Y_1(-\Delta+I)^{-1/2}](-\Delta+I)^{1/2}(D_m-\lambda I)^{-1}W_2.
\end{eqnarray*} It follows from H\"older's inequality and Corollary ~\ref{corollary of three cases} that
\begin{eqnarray*}
    \|\tilde X_{1,2}\|_{2/3}&\le&C\|W_1(D_m-\lambda I)^{-1}\chi_\delta\|_{\Sigma_1}\cdot\|(-\Delta+I)^{1/2}(D_m-\lambda I)^{-1}W_2\|_{\Sigma_2}\\
    &\le&CC_0^{1/2}[W_1^2]^{1/2}_1[\chi_\delta]_1^{1/2}\|(-\Delta+I)^{1/2}(D_m-\lambda I)^{-1}W_2\|_{\Sigma_2}\\
    &=&O(\alpha^{\frac{3}{2p}-\frac{1}{2}})\|(-\Delta+I)^{1/2}(D_m-\lambda I)^{-1}W_2\|_{\Sigma_2},
\end{eqnarray*} for some constant $C\ge0$. To approximate the last factor on the right-hand side, we use the representation
\begin{eqnarray*}
    (-\Delta+I)^{1/2}(D_m-\lambda I)^{-1}W_2&=&F^*\frac{(|\xi|^2+1)^{1/2}}{m^2-\lambda^2+|\xi|^4}\begin{pmatrix}
        \lambda+m&-(\xi_1-i\xi_2)^2\\-(\xi_1+i\xi_2)^2&\lambda-m
    \end{pmatrix}^2FW_2\\
    &=&F^*\frac{1}{(|\xi|^2+1)^{1/2}}\tilde b(\xi)FW_2,
\end{eqnarray*} for some $\tilde b\in L^\infty(\mathbb R^2)$. Hence,
\begin{align*}
&\|(-\Delta+I)^{1/2}(D_m-\lambda I)^{-1}W_2\|_{\Sigma_2}
=\|F^*\frac{1}{(|\xi|^2+1)^{1/2}}\tilde b(\xi)FW_2\|_{\Sigma_2}
\\&\le\|\tilde b\|_\infty\cdot\|(|\xi|^2+1)^{-{1/2}}FW_2\|_{\Sigma_2}=\|\tilde b\|_\infty\cdot\|W_2F^*\frac{1}{\sqrt{|\xi|^2+1}}\frac{1}{\sqrt{|\xi|^2+1}}FW_2\|_{\Sigma_1}^{1/2}\\
&=\|\tilde b\|_\infty\|W_2(-\Delta+I)^{-1}W_2\|_{\Sigma_1}^{1/2}\le C_5[W_2^2]_1^{1/2}\\
&=C_5\sum_{n\in\mathbb Z^d}\|W_2^2\|^{1/2}_{L^q(\mathbb Q_n)}
=O(\alpha^{\frac{1}{p}-\frac{1}{2}}),\quad\text{as }\alpha\to\infty,
\end{align*} where $C_5$ is a non-negative constant. Therefore, we get
\begin{eqnarray*}
    \|\tilde X_{1,2}\|_{\Sigma_{2/3}}=O(\alpha^{\frac{5}{2p}-1}),\quad\text{as }\alpha\to\infty,
\end{eqnarray*} which leads to the inequality
\begin{eqnarray*}
    (\varepsilon\alpha^{-1})^{2/3}n(\varepsilon\alpha^{-1},\tilde X_{1,2})\le\|\tilde X_{1,2}\|^{2/3}_{\Sigma_{2/3}}=O(\alpha^{\frac{5}{3p}-\frac{2}{3}}).
\end{eqnarray*} As a result, we establish that $n(\varepsilon\alpha^{-1},\tilde X_{1,2})=O(\alpha^{\frac{5}{3p}})=o(\alpha^{2/p})$. The analysis of operator $\tilde X_{1,2}$ is accomplished.

We still have to demonstrate that $n(\varepsilon\alpha^{-1},\hat X_{1,2})=o(\alpha^{2/p})$  as $\alpha\to\infty$. Note that
\begin{eqnarray*}
    \hat X_{1,2}=W_1(D_m-\lambda I)^{-1}(1-\theta)W_2
    =W_1F^*\frac{1}{(\xi^2+1)^{1/2}}b(\xi)\frac{1}{(\xi^2+1)^{1/2}}F(1-\theta)W_2.
\end{eqnarray*} As a consequence,
\begin{eqnarray*}
    \|\hat X_{1,2}\|_{\Sigma_1}&\le&\|b\|_\infty\|W_1F^*\frac{1}{(\xi^2+1)^{1/2}}\|_{\Sigma_2}\cdot\|\frac{1}{(\xi^2+1)^{1/2}}F(1-\theta)W_2\|_{\Sigma_2}\\
    &\le&C_0\|b\|_\infty[W_1^2]_1^{1/2}[(1-\theta)^2W_2^2]_1^{1/2}.
\end{eqnarray*} Recall that we have already illustrated the relation
\begin{eqnarray*}
    [W_1^2]_1=O(\alpha^{\frac{2}{p}-1}),\quad\text{as }\alpha\to\infty.
\end{eqnarray*} Moreover,
\begin{eqnarray*}
    [(1-\theta)^2W_2^2]_1&=&\sum_{n\in\mathbb Z^d}\|(1-\theta)^2W_2^2\|_{L^q(\mathbb Q_n)}\asymp C\int_{\varepsilon_1\alpha^{1/p}<r<\varepsilon_1\alpha^{1/p}+\delta}\frac{r\,dr}{r^p}\\
    &=&C\left[\frac{r^{2-p}}{2-p}\right]\Big|^{r=\varepsilon_1\alpha^{1/p}+\delta}_{r=\varepsilon_1\alpha^{1/p}}    =C((\varepsilon_1\alpha^{1/p}+\delta)^{2-p}-(\varepsilon_1\alpha^{1/p})^{2-p})\\
    &=&(\varepsilon_1\alpha^{1/p})^{2-p}((1+\frac{\delta}{\varepsilon_1\alpha^{1/p}})^{2/p}-1)\sim(\varepsilon_1\alpha^{1/p})^{2-p}\left(\frac{\delta(2-p)}{\varepsilon_1\alpha^{1/p}}\right)\\&=&O(\alpha^{\frac{1}{p}-1}).
\end{eqnarray*} Thus,
\begin{eqnarray*}
    \|\hat X_{1,2}\|_{\Sigma_1}=O(\alpha^{\frac{3}{2p}-1}),
\end{eqnarray*} which implies
\begin{eqnarray*}
    n(\varepsilon\alpha^{-1},\hat X_{1,2})=O(\alpha^\frac{3}{2p})=o(\alpha^{\frac{2}{p}}),\quad\text{as }\alpha\to\infty.
\end{eqnarray*}
 
\subsection{Operator $W_2(D_0-\lambda I)^{-1}W_3$}

We employ the same approach as in the preceding subsection to examine the second non-self-adjoint operator in \eqref{decomposition}, specifically, $W_2(D_0-\lambda I)^{-1}W_3$.

Let function $\zeta\in C^\infty(\mathbb R)$ be defined as in \eqref{zeta} and function $\eta:\mathbb R^2\to\mathbb R$ be given by $\eta(x)=\zeta(|x|-\varepsilon_2\alpha^{1/p})$. There is no doubt that \eqref{another form of the first operator}, \eqref{commutator result} and \eqref{commutator} would all remain applicable if we replace $\theta$, $W_1$ and $W_2$ by $\eta,W_2$ and $W_3$, respectively. In particular,
\begin{eqnarray}
    W_2(D_m-\lambda I)^{-1}W_3&=W_2(D_m-\lambda I)^{-1}\eta W_3+ W_2(D_m-\lambda I)^{-1}(1-\eta) W_3\nonumber\\
    =&W_2[(D_m-\lambda I)^{-1},\eta]\,W_3+ W_2(D_m-\lambda I)^{-1}(1-\eta) W_3\label{commutator for W2W3},
\end{eqnarray} and the first term on the right-hand side can be written as
\begin{eqnarray}
\label{Y0 and Y1 for the second operator}
    W_2[(D_m-\lambda I)^{-1},\eta]\,W_3&=&W_2(D_m-\lambda I)^{-1}(\tilde Y_0+\tilde Y_1)(D_m-\lambda I)^{-1}W_3 \nonumber\\
    &=&X_{2,3}+\tilde X_{2,3},
\end{eqnarray} where
\begin{eqnarray*}
    X_{2,3}&=&W_2(D_m-\lambda I)^{-1}\tilde Y_0(D_m-\lambda I)^{-1}W_3,\\
    \tilde X_{2,3}&=&W_2(D_m-\lambda I)^{-1}\tilde Y_1(D_m-\lambda I)^{-1}W_3\\
    \tilde Y_0&=&\begin{pmatrix}
        0&\eta_1\\\eta_2&0
    \end{pmatrix},
    \tilde Y_1=2\begin{pmatrix}
        0&\eta_3(\frac{\partial}{\partial x_1}-i\frac{\partial}{\partial x_2})\\\eta_4        (\frac{\partial}{\partial x_1}+i\frac{\partial}{\partial x_2})&0    \end{pmatrix},
\end{eqnarray*} with $\eta_1,\eta_2,\eta_3$ and $\eta_4$ defined in the same way as $\theta_1,\theta_2,\theta_3$ and $\theta_4$, respectively, except that, each time, 
$\theta$ is replaced by $\eta$. We argue in the same way as before to show that $n(s,X_{2,3})=o(\alpha^{2/p})$ and $n(s,\tilde X_{2,3})=o(\alpha^{2/p})$.

Let  us  first deal   with the  operator $X_{2,3}$.  Denote  by
$\psi_\delta$ the characteristic function of set $\{x\in\mathbb R^2:\varepsilon_2\alpha^{1/p}<|x|<\varepsilon_2\alpha^{1/p}+ \delta\}$. Then
\begin{eqnarray*}
X_{2,3}=W_2(D_m-\lambda I)^{-1}\psi_\delta  \tilde Y_0\psi_\delta(D_m-\lambda I)^{-1}W_3.
\end{eqnarray*} Referring to  the bound $\|\tilde Y_0\|\le\|\eta_1\|_\infty+\|\eta_2\|_\infty$ and  using H\"older's inequality in the $\Sigma_p$-class, we obtain
\begin{eqnarray}
\label{holders inequality}
    \|X_{2,3}\|_{\Sigma_1}\le(\|\eta_1\|_\infty+\|\eta_2\|_\infty)\|W_2(D_m-\lambda I)^{-1}\psi_\delta\|_{\Sigma_1}\cdot\|\psi_\delta(D_m-\lambda I)^{-1}W_3\|.
\end{eqnarray} If  $b\in L^\infty(\mathbb R^2)$ is the function  specified in the preceding  subsection,  then
\begin{align}
    \|W_2(D_m-\lambda I)^{-1}\psi_\delta\|_{\Sigma_1}&\le\|W_2F^*(|\xi|^2+1)^{-1/2}b(\xi)(|\xi|^2+1)^{-1/2}F\psi_\delta\|_{\Sigma_1}\nonumber\\
    \le&\|b\|_\infty\|W_2F^*(|\xi|^2+1)^{-\frac{1}{2}}\|_{\Sigma_2}\|(|\xi|^2+1)^{-\frac{1}{2}}F\psi_\delta\|_{\Sigma_2}\label{initial estimate 2}.
\end{align} Furthermore, \eqref{one estimate} and \eqref{another estimate} hold  when $W_1$ and $\chi_\delta$ are substituted by $W_2$ and $\psi_\delta$, respectively. Combining  them yields the following estimate:
\begin{eqnarray*}
    \|W_2(D_m-\lambda I)^{-1}\psi_\delta\|_{\Sigma_2}&\le& C_8\|b\|_\infty[\psi_\delta]^{1/2}_1\cdot[W_2^2]^{1/2}_1.
\end{eqnarray*} On the other hand, 
\begin{eqnarray*}
    \sum_{n\in\mathbb Z^d}\|W_2^2\|_{L^q(\mathbb Q_n)}&\le& C_9\sum_{n\in\mathbb Z^d,\varepsilon_1\alpha^{1/p}<|n|<\varepsilon_2\alpha^{1/p}}\frac{1}{|n|^p+1}\\
    &\le&C_{10}\int_{\varepsilon_1\alpha^{1/p}<|x|<\varepsilon_2\alpha^{1/p}}\frac{dx}{|x|^p+1}\asymp C_{10}\int_0^{2\pi}\int^{\varepsilon_1\alpha^{1/p}}_{\varepsilon_2\alpha^{1/p}}r^{1-p}\,dr\,d\theta\\
    &=&2\pi C_{10}\left[\frac{r^{2-p}}{2-p}\right]\Big|_{r=\varepsilon_1\alpha^{1/p}}^{r=\varepsilon_2\alpha^{1/p}}=C_{11}(\varepsilon_2^{2-p}-\varepsilon_1^{2-p})\alpha^{\frac{2}{p}-1},
\end{eqnarray*} where $C_9,C_{10}$ and $C_{11}$ are all non-negative constants. Also, we have
\begin{eqnarray*}
    [\psi_\delta]_1&=&\sum_{n\in\mathbb Z^d}\|\psi_\delta\|_{L^q(\mathbb Q_n)}\le\sum_{\varepsilon_2\alpha^{1/p}<|n|<\varepsilon_2\alpha^{1/p}+\delta}1\\
    &=&\text{Area}\{x\in\mathbb R^2:\varepsilon_2\alpha^{1/p}<|x|<\varepsilon_2\alpha^{1/p}+\delta\}=\pi[(\varepsilon_2\alpha^{1/p}+\delta)^2-(\varepsilon_2\alpha^{1/p})^2]\\
    &=&\pi(2\varepsilon_2\alpha^{1/p}\delta+\delta^2).
\end{eqnarray*} As a consequence,
\begin{eqnarray*}
    \|W_2(D_m-\lambda I)^{-1}\psi_\delta\|_{\Sigma_1}\le O(\alpha^{\frac{1}{p}-\frac{1}{2}})O(\alpha^{\frac{1}{2p}})=O(\alpha^{\frac{3}{2p}-\frac{1}{2}}),\qquad\text{as } \alpha\to\infty.
\end{eqnarray*} Therefore,  since according to \eqref{holders inequality},  there is  a constant $C_{12}\ge0$  for which
\begin{eqnarray*}
    \|X_{2,3}\|_{\Sigma_1}\le C_{12}\|W_2(D_m-\lambda I)^{-1}\psi_\delta\|_{\Sigma_1}\cdot\|\psi_\delta(D_m-\lambda I)^{-1}W_3\|,
\end{eqnarray*}  it remains to estimate the last factor on the right hand side.  Recall that for every $\lambda\in(-m,m)$,
\begin{eqnarray*}
    \|(D_m-\lambda I)^{-1}\|\le\frac{1}{\text{dist}(\lambda,\sigma(D_m))},
\end{eqnarray*} where $\sigma(D_m)$ denotes the spectrum of the operator $D_m$. Also, there is a constant $C$ that depends only on $\lambda$ and satisfies that $W_3(x)\le C|x|^{-p/2}$, which implies that
\begin{eqnarray*}
    \|\psi_\delta(D_m-\lambda I)^{-1}W_3\|\le \frac{C}{\text{dist}(\lambda,\sigma(D_m))}\varepsilon_2^{-p/2}\alpha^{-1/2}.
\end{eqnarray*} As a consequence, we obtain the following formula:
\begin{eqnarray*}
    \varepsilon\alpha^{-1}n(\varepsilon\alpha^{-1},X_{2,3})\le\|X_{2,3}\|_{\Sigma_{1}}=O(\alpha^{\frac{3}{2p}-1}),
\end{eqnarray*} which leads to
\begin{eqnarray*}
    n(\varepsilon\alpha^{-1},X_{2,3})= O(\alpha^{\frac{3}{2p}})=o(\alpha^{2/p}),\quad\text{as }\alpha\to\infty.
\end{eqnarray*} That is exactly the desired relation.\\

We now study operator $\tilde X_{2,3}$. Observe that an alternative expression for $\tilde X_{2,3}$ is
\begin{eqnarray*}
    \tilde X_{2,3}=W_2(D_m-\lambda I)^{-1}\psi_\delta[\tilde Y_1(-\Delta+I)^{-1/2}](-\Delta+I)^{1/2}(D_m-\lambda I)^{-1}W_3.
\end{eqnarray*} Using boundedness of the operator $\tilde Y_1(-\Delta+I)^{-1/2}$, H\"older's inequality and Corollary ~\ref{corollary of three cases}, we estimate the $\Sigma_1$-norm of $\tilde X_{2,3}$ as follows:
\begin{eqnarray*}
    \|\tilde X_{2,3}\|_{\Sigma_1}&\le&C_{13}\|W_2(D_m-\lambda I)^{-1}\psi_\delta\|_{\Sigma_1}\cdot\|(-\Delta+I)^{1/2}(D_m-\lambda I)^{-1}W_3\|\\
    &\le&C_{13}C_0[W_2^2]^{1/2}_1[\psi_\delta]_1^{1/2}\|(-\Delta+I)^{1/2}(D_m-\lambda I)^{-1}W_3\|\\
    &=&O(\alpha^{\frac{3}{2p}-\frac{1}{2}})\|(-\Delta+I)^{1/2}(D_m-\lambda I)^{-1}W_3\|,
\end{eqnarray*} with some constant $C_{13}\ge0$. Also, note that
\begin{eqnarray*}
    (-\Delta+I)^{1/2}(D_m-\lambda I)^{-1}W_3&=&F^*\frac{(|\xi|^2+1)^{1/2}}{m^2-\lambda^2+|\xi|^4}\begin{pmatrix}
        \lambda+m&-(\xi_1-i\xi_2)^2\\-(\xi_1+i\xi_2)^2&\lambda-m
    \end{pmatrix}^2FW_3\\
    &=&F^*\frac{1}{(|\xi|^2+1)^{1/2}}\tilde b(\xi)FW_3,
\end{eqnarray*} where $\tilde b\in L^\infty(\mathbb R^2)$. By the fact that there is some constant $C\ge 0$ such that $|W_3(x)|$ is bounded from above by $C|x|^{-p/2}$, we get
\begin{eqnarray*}
    \|F^*\frac{1}{(|\xi|^2+1)^{1/2}}\tilde b(\xi)FW_3\|\le\|\tilde b\|_\infty\cdot\|W_3\|_\infty\le C_{14}\|\tilde b\|_\infty \varepsilon_2^{-p/2}\alpha^{-1/2}.
\end{eqnarray*} Therefore,
\begin{eqnarray*}
    \|\tilde X_{2,3}\|_{\Sigma_1}\le O(\alpha^{\frac{3}{2p}-1}),\quad\text{as }\alpha\to\infty,
\end{eqnarray*} which gives
\begin{eqnarray*}
    \varepsilon\alpha^{-1}n(\varepsilon\alpha^{-1},\tilde X_{2,3})\le\|\tilde X_{2,3}\|_{\Sigma_1}=O(\alpha^{\frac{3}{2p}-1}).
\end{eqnarray*} This is enough to conclude that $n(s,\tilde X_{2,3})=O(\alpha^{3/2p})=o(\alpha^{2/p})$. By this  we complete   the work with the operator $\tilde X_{2,3}$.\\

However,  according to \eqref{commutator for W2W3}, we still have to show that $n(\varepsilon\alpha^{-1},\hat X_{2,3})=o(\alpha^{2/p})$,  as $\alpha\to\infty$, for the operator
$$
 \hat X_{2,3}=W_2(D_m-\lambda I)^{-1}(1-\eta)W_3.
$$
 For this purpose, we observe that
\begin{eqnarray*}
    \hat X_{2,3}&=&W_2(D_m-\lambda I)^{-1}(1-\eta)W_3\\
    &=&W_2F^*\frac{1}{(\xi^2+1)^{1/2}}b(\xi)\frac{1}{(\xi^2+1)^{1/2}}F(1-\eta)W_3.
\end{eqnarray*} As a consequence, we obtain
\begin{eqnarray*}
    \|\hat X_{2,3}\|_{\Sigma_1}&\le&\|b\|_\infty\|W_2F^*\frac{1}{(\xi^2+1)^{1/2}}\|_{\Sigma_2}\cdot\|\frac{1}{(\xi^2+1)^{1/2}}F(1-\eta)W_3\|_{\Sigma_2}\\
    &\le&C_0\|b\|_\infty[W_2^2]_1^{1/2}[(1-\eta)^2W_3^2]_1^{1/2}.
\end{eqnarray*} Now we recall that 
\begin{eqnarray*}
    [W_2^2]_1=O(\alpha^{\frac{2}{p}-1}),\quad\text{as }\alpha\to\infty.
\end{eqnarray*} Furthermore,
\begin{align*}
    &[(1-\eta)^2W_3^2]_1=\sum_{n\in\mathbb Z^d}\|(1-\eta)^2W_3^2\|_{L^q(\mathbb Q_n)}\asymp C\int_{\varepsilon_2\alpha^{1/p}<r<\varepsilon_2\alpha^{1/p}+\delta}\frac{r\,dr}{r^p}\\
    &=C\left[\frac{r^{2-p}}{2-p}\right]\Big|^{r=\varepsilon_2\alpha^{1/p}+\delta}_{r=\varepsilon_2\alpha^{1/p}}=C((\varepsilon_2\alpha^{1/p}+\delta)^{2-p}-(\varepsilon_2\alpha^{1/p})^{2-p})=O(\alpha^{\frac{1}{p}-1}).
\end{align*} Therefore,
\begin{eqnarray*}
    \|\hat X_{2,3}\|_{\Sigma_1}=O(\alpha^{\frac{3}{2p}-1}),
\end{eqnarray*} which implies that
\begin{eqnarray*}
    n(\varepsilon\alpha^{-1},\hat X_{2,3})=O(\alpha^\frac{3}{2p})=o(\alpha^{\frac{2}{p}}),\quad\text{as }\alpha\to\infty.
\end{eqnarray*}

\subsection{Operator $W_1(D_m-\lambda I)^{-1}W_3$}

Applying a similar method, we demonstrate that $n(\varepsilon\alpha^{-1},W_1(D_m-\lambda I)^{-1}W_3)=o(\alpha^{2/p})$ as $\alpha\to\infty$ for every $\varepsilon>0$.\\

The decomposition of the space $L^2({\Bbb R}^2)$ into the orthogonal sum $L^2(\Omega_1(\alpha))\oplus L^2(\Omega_2(\alpha))\oplus L^2(\Omega_3(\alpha))$
leads to the  decomposition of the  operator $X_\lambda$ formally displayed   by the matrix
\begin{equation} \notag
X_\lambda= \left( \begin{array}{ccc} 
W_1(D_m-\lambda I)^{-1}W_1 & W_1(D_m-\lambda I)^{-1}W_2& W_1(D_m-\lambda I)^{-1}W_3  \\ [0.5cm]
W_2(D_m-\lambda I)^{-1}W_1&W_2(D_m-\lambda I)^{-1}W_2& W_2(D_m-\lambda I)^{-1}W_3 \\ [0.5cm]
W_3(D_m-\lambda I)^{-1}W_1  & W_3(D_m-\lambda I)^{-1}W_2& W_3(D_m-\lambda I)^{-1}W_3 
\end{array}\right).
\end{equation} 
We  have already proved that the off-diagonal   elements of this matrix  do not contribute  to the asymptotics of $N(\lambda,\alpha)$.
It remains to compute the contribution of    the diagonal elements,   whose orthogonal sum is denoted by $X_\lambda^+$:
\begin{equation} \notag
X^+_\lambda= \left( \begin{array}{ccc} 
W_1(D_m-\lambda I)^{-1}W_1 & 0& 0  \\ [0.5cm]
0&W_2(D_m-\lambda I)^{-1}W_2& 0 \\ [0.5cm]
0 & 0& W_3(D_m-\lambda I)^{-1}W_3 
\end{array}\right).
\end{equation} 
Namely,  according to Proposition ~\ref{integral}, we need to show that
\begin{eqnarray*}
    n_+(\tau\varepsilon^{-1}, X_\lambda^+)\sim\tau^{-2/p}\alpha^{2/p}J(\lambda,m),\quad\text{as }\alpha\to\infty,
\end{eqnarray*} for every $\tau>0$.

\subsection{Operator $W_1(D_m-\lambda I)^{-1}W_1$}

We start  the  investigation of  the first self-adjoint operator in \eqref{decomposition}, $W_1(D_m-\lambda I)^{-1}W_1$.

First, recall that, by \eqref{very important result}, the symbol $F(D_m-\lambda I)^{-1}F^*$ decays as $O(|\xi|^{-2})$ when $|\xi|\to\infty$. Hence, there exists a function $b\in L^\infty(\mathbb R^2)$ satisfying
\begin{eqnarray*}
    F(D_m-\lambda I)^{-1}F^*= \frac{b(\xi)}{1+|\xi|^2},
\end{eqnarray*} which implies the bound
\begin{eqnarray}
    \|W_1(D_m-\lambda I)^{-1}W_1\|_{\Sigma_1}&\le\|W_1F^*(|\xi|^2+1)^{-1/2}b(\xi)(|\xi|^2+1)^{-1/2}FW_1\|_{\Sigma_1}\nonumber\\
    \le&\|b\|_\infty\|W_1F^*(|\xi|^2+1)^{-\frac{1}{2}}\|_{\Sigma_2}\|(|\xi|^2+1)^{-\frac{1}{2}}FW_1\|_{\Sigma_2}\label{initial estimate 2}.
\end{eqnarray} On the other hand, based on Corollary ~\ref{corollary of three cases}, we obtain
\begin{eqnarray}
    &\|W_1F^*(|\xi|^2+1)^{-1/2}\|^2_{\Sigma_2}=\|W_1F^*(|\xi|^2+1)^{-1/2}(|\xi|^2+1)^{-1/2}FW_1\|_{\Sigma_1}\nonumber\\
    &=\|W_1F^*(|\xi|^2+1)^{-1}FW_1\|_{\Sigma_1}=\|W_1(-\Delta+I)^{-1}W_1\|_{\Sigma_1}\label{initial estimate 3}
    \le C_0[W_1^2]_1.
\end{eqnarray} Therefore, 
\begin{align*}
    &\|W_1(D_m-\lambda I)^{-1}W_1\|_{\Sigma_1}\le C_0\|b\|_\infty[W_1^2]_1=C_0\|b\|_\infty\sum_{n\in\mathbb Z^d}\|W_1^2\|_{L^q(\mathbb Q_n)}\\
    &\le C_{15}\sum_{n\in\mathbb Z^d,|n|<\varepsilon_1\alpha^{1/p}}\frac{1}{|n|^p+1}\le C_{16}\int_{|x|<\varepsilon_1\alpha^{1/p}}\frac{dx}{|x|^p+1}\\
    &\asymp C_{16}\int_0^{2\pi}\int_{r<\varepsilon_1\alpha^{1/p}}r^{1-p}\,dr\,d\theta=2\pi C_{16}\left[\frac{r^{2-p}}{2-p}\right]\Big|_{r=0}^{r=\varepsilon_1\alpha^{1/p}}=C_{17}\varepsilon_1^{2-p}\alpha^{\frac{2}{p}-1},
\end{align*} where $C_{15},C_{16}$ and $C_{17}$ are non-negative constants. This suggests that
\begin{eqnarray}
\label{result for the first self adjoint operator}
    n(\tau\alpha^{-1},W_1(D_m-\lambda I)^{-1}W_1)\le C_{18}\varepsilon_1^{2-p}\alpha^{2/p}\tau^{-1},\qquad\forall\tau>0,
\end{eqnarray} for some constant $C_{18}\ge0$. This concludes the analysis of the operator $W_1(D_m-\lambda I)^{-1}W_1$.

\subsection{Operator $W_3(D_m-\lambda I)^{-1}W_3$}

Let us show that 
\begin{align}
\label{tau}
    \|W_3(D_m-\lambda I)^{-1}W_3\|\le\tau\alpha^{-1},
\end{align} if $\varepsilon_2$ is sufficiently large. This would imply that
\begin{align}
\label{result for W3W3}
    n(\tau\alpha^{-1},W_3(D_m-\lambda I)^{-1}W_3)=0.
\end{align} Note that there is a constant $C\ge 0$ such that $|W_3(x)|$ is bounded from above by $C|x|^{-p/2}$. Consequently, 
\begin{eqnarray*}
    &\|W_3(D_m-\lambda I)^{-1}W_3\|\le\frac{1}{\text{dist}(\lambda,\sigma(D_m))}\|W_3^2\|_\infty=C_\lambda\varepsilon_2^{-p}\alpha^{-1},
\end{eqnarray*} with the constant
\begin{align*}
    C_\lambda=\frac{C}{\text{dist}(\lambda,\sigma(D_m))}.
\end{align*} Hence, if we choose $\varepsilon_2>(C_\lambda/\tau)^{1/p}$, then \eqref{tau} will hold. The analysis of the operator $W_3(D_m-\lambda I)^{-1}W_3$ is thereby completed.

\subsection{Operator $W_2(D_0-\lambda I)^{-1}W_2$}

We now consider the only remaining self-adjoint operator in \eqref{decomposition}, $W_2(D_0-\lambda I)^{-1}W_2$. To get started, we introduce the notation $\tilde\Omega_2=\{x\in\mathbb R^2:\varepsilon_1<|x|<\varepsilon_2\}$. Thus, $\Omega_2(\alpha)=\alpha^{1/p}\tilde\Omega_2$, which means  that one   set is obtained from the other  by scaling. Then we decompose $\tilde\Omega_2$ into finitely many sets denoted by $\{Q_j\}_{j=1}^l$, where $Q_j$ is a square for $1\le j\le l-1$ and $Q_l=\tilde\Omega_2\backslash\cup_{j=1}^{l-1}Q_j$. Evidently, each square $Q_j$ can be expressed in the form 
\begin{align}
    \delta([0,1)^2+n)
    \label{square}
\end{align} for some $\delta>0$ and $n\in\mathbb Z^2$  when $1\le j\le l-1$. Moreover, $\Omega_2(\alpha)=\alpha^{1/p}\tilde\Omega_2=\cup_{j=1}^l\alpha^{1/p}Q_j$. Employing the same methodology as in the proof of Proposition ~\ref{integral} and the relation \eqref{result of KF inequality}, we can show that
\begin{align}
\label{sum estimate}
    &n_+(\tau\alpha^{-1},W_2(D_m-\lambda I)^{-1}W_2)\nonumber\\&\sim\sum_{j=1}^l n_+(\tau\alpha^{-1},\chi_jW(D_m-\lambda I)^{-1}W\chi_j)+o(\alpha^{2/p}),\quad\text{as }\alpha\to\infty,
\end{align} where $\chi_j$ represents the characteristic function of set $\alpha^{1/p}Q_j$ for $1\le j\le l$. In order to compute the asymptotics of $n_+(\tau\alpha^{-1},W_2(D_m-\lambda I)^{-1}W_2)$, it suffices to analyze each individual term on the right-hand side of \eqref{sum estimate}.

Let $Q$ be a square of the form \eqref{square}, let $\phi_\beta$ be the characteristic function of the set $\beta Q$. We want to investigate the behavior of the quantity $n_+(\tau,\phi_\beta(D_m-\lambda I)^{-1}\phi_\beta)$ when the value of $\beta$ approaches infinity and $\tau>0$ is fixed. We intend  to calculate the value of  the limit $\lim_{\beta\to\infty}\beta^{-2}n_+(\tau,\phi_\beta(D_m-\lambda I)^{-1}\phi_\beta)$.

Note that for each self-adjoint operator $A$, we have $\chi_{(\tau,\infty)}(A)=E_A(\tau,\infty)$. Thus,
\begin{eqnarray*}
    n_+(\tau,A)=\text{tr}(E_A(\tau,\infty))=\text{tr}(\chi_{(\tau,\infty)}(A)),
\end{eqnarray*} if $A$ is compact and self-adjoint. 

We are going to prove the following important proposition.

\begin{proposition}\label{pr2.2}
For every fixed $\t>0$, 
$$
n_+\bigl(\tau, \phi_{\beta}(D_m-\l)^{-1}\phi_{\beta}\bigr)\sim (4\pi)^{-1}\beta^{2}\Bigl((\t^{-1}+\l)_+^2-m^2\Bigr)_+^{1/2} {\rm Area }\, Q, \quad \text{as }\beta\to\infty,
$$ where $x_+=\max(0,x)$ for all $x\in\mathbb R$.
\end{proposition}

\begin{proof} First, note that
$$
\pi\Bigl((\t^{-1}+\l)_+^2-m^2\Bigr)_+^{1/2}={\rm area}\,\{\xi\in {\Bbb R}^2:\,\,\,\,\,   \bigl(\sqrt{|\xi|^4+m^2}-\l\bigr)^{-1}>\t\}.
$$
Based on the fact that $\pm \sqrt{|\xi|^4+m^2}$ are eigenvalues of the symbol 
$$\hat D_m(\xi):=\begin{pmatrix} m & (\xi_2-i\xi_1)^2 \\
(\xi_2+i\xi_1)^2 & -m
\end{pmatrix}$$ we find
\[\label{2.1}
{\rm tr}\, \Phi \bigl(\phi_{\beta}(D_m-\l I)^{-1}\phi_{\beta}\bigr)\sim {\rm tr}\,\Bigl[\phi_{\beta}\Phi \bigl((D_m-\l I)^{-1}\bigr)\phi_{\beta}\Bigr],\quad\text{as } \beta\to\infty,
\]
for $\Phi$ being the characteristic function of the interval $(\t,\infty)$. It is clear that such a function $\Phi$ could be squeezed between two comparable functions of the form
$$\Phi_\varepsilon(s)= \begin{cases} 0, \quad &\text{if }s<s_0\\
(s-s_0)/\varepsilon; \quad&\text{if } s_0 \leq s\leq s_0+\varepsilon;\\1, \quad&\text{if } s>s_0+\varepsilon.
 \end{cases}$$
Namely,
\[\notag
\Phi_\varepsilon(s)\leq \Phi(s)\leq \Phi_\varepsilon(s+\varepsilon).
\]
Therefore, if \eqref{2.1} holds with $\Phi$ replaced by the functions $\Phi_\varepsilon(s)$ and $\Phi_\varepsilon(s+\varepsilon)$, then
\[\notag
\Bigl(((\t+\varepsilon)^{-1}+\l)_+^2-m^2\Bigr)_+^{1/2}\leq
\lim_{\beta\to\infty}\beta^{-2}
{\rm tr}\, \Phi \bigl(\phi_{\beta}(D_m-\l I)^{-1}\phi_{\beta}\bigr)\leq \Bigl(((\t-\varepsilon)^{-1}+\l)_+^2-m^2\Bigr)_+^{1/2}.
\]
Since $\varepsilon>0$ is arbitrary, the latter inequality would result in \eqref{2.1}. Thus, we only have to prove \eqref{2.1} for functions $\Phi$ that are continuous and vanish near the origin.

Every such function $\Phi$ may be rewritten as
$$
\Phi(s)=s^{2}\eta(s),
$$
where $\eta$ is a continuous function on the real line ${\Bbb R}$. Notice that, in this situation,
$$
\Bigl|{\rm tr}\, \Phi \bigl(\phi_{\beta}(D_m-\l I)^{-1}\phi_{\beta}\bigr)\Bigr|\leq \| \phi_{\beta}(D_m-\l I)^{-1}\phi_{\beta}\|_{{\frak S}_2}^2 \| \eta\|_\infty\leq C_{20}\beta^{2}\| \eta\|_\infty,
$$ for some constant $C_{20}\ge0$. By $\Phi((D_m-\l I)^{-1})=(D_m-\l I)^{-1}\eta(R_\l I)(D_m-\l I)^{-1}$, we derive that
$$
\Bigl|{\rm tr}\, \phi_{\beta}\Phi \bigl((D_m-\l I)^{-1}\bigr)\phi_{\beta}\Bigr|\leq \| \phi_{\beta}(D_m-\l I)^{-1}\|_{{\frak  S}_{2}}  \| (D_m-\l I)^{-1}\phi_{\beta}\|_{{\frak  S}_{2}}  \| \eta\|_\infty\leq C_{20}\beta^{2}\| \eta\|_\infty.
$$
Consequently, both sides of \eqref{2.1} may be estimated by $C_{20}\beta^{2}\|\eta\|_\infty$, allowing one to assume that $\eta$ is a polynomial.

Indeed, functions of a given self-adjoint operator only need to be defined on the spectrum of the operator. On the other hand, the spectrum of operator $(D_m-\l I)^{-1}$ is contained in $[-L,L]$, where $L=1/(m-|\l|)$.
As a result, the functional $\|\eta\|_\infty$ in the last inequality is the $L^\infty$-norm of the function on the compact interval $[-L,L]$. Since on a finite interval, $\eta$ can be uniformly approximated by polynomials, it suffices to prove \eqref{2.1} under the assumption that $\eta$ is a polynomial. Put differently, it is enough to show that \eqref{2.1} holds for
$$
\Phi(s)=s^{k},\qquad k\geq2,
$$
by the fact that all polynomials are finite linear combinations of power functions.

Denote $R_\l=(D_m-\l I)^{-1}$,  $\chi_+=\phi_{\beta}$ and $\chi_-=1-\phi_{\beta}.$
We intend to show that
\[\label{2.2}
\|(\chi_+R_\l\chi_+)^k-\chi_+R_\l^k\chi_+\|_{{\frak S}_1}=o(\beta^{2}),\qquad \text{ as} \quad \beta\to\infty.
\]
For that reason, we represent $\chi_+R_\l^k\chi_+$ as
\[\label{2.3}
\chi_+R_\l^k\chi_+=(\chi_+R_\l\chi_+)^k+\sum_{j=0}^{k-1} (\chi_+R_\l\chi_+)^j\chi_+R_\l\chi_-R_\l^{k-j-1}\chi_+.
\]
While the norm of the operator $\chi_+R_\l\chi_-$ does not approach zero, it is still representable in the form
$$
\chi_+R_\l\chi_-= T_1+T_2,\qquad \text{where } \|T_1\|\to 0,  \|T_2\|_{{\frak S}_{k}}=o(\beta^{2/k}),\quad \text{as } \beta\to\infty.
$$ Also, we define $T_2$ as
$$
T_2=\Theta\chi_+R_\l\chi_-\Theta,
$$
where $\Theta$ is the operator of multiplication by the characteristic  function of the layer
\[
\label{layer1}
\{x\in {\Bbb R}^3:\quad {\rm dist}(x,\beta \partial Q)<\beta^{1/2}\}.
\]
Then the area of the support of $\Theta$ does not exceed $C\beta^{3/2}$ at least for sufficiently large values of $\beta$. Thus,
$$
 \|T_2\|_{{\frak S}_{k}}\leq C\beta^{\frac{3}{2k} }=o(\beta^{2/k}),\quad \text{as}\quad \beta\to\infty.
$$
In contrast, since the explicit expression  for the integral kernel of  $(D_m-\l I)^{-1}$ shows that the latter decays exponentially fast when $|x-y|\to\infty$, we conclude the following estimate for the integral kernel $k(x,y)$ of the operator $T_1$:
\[
|k(x,y)|\leq C \bigl((1-\Theta(x))+(1-\Theta(y))\bigr)e^{-c|x-y|}\chi_+(x)\chi_-(y).
\]
Then combining the celebrated Shur estimate
$$
\|T_1\|\leq \left(\sup_x \int |k(x,y)|dy\times \sup_y \int | k(x,y)|dx\right)^{1/2}
$$  with the observation that $x$ and $y$ for which $k(x,y)\neq 0$ are distance  $\beta^{1/2}$ apart, we obtain
the inequality
$$
\|T_1\|\leq 2\pi C \int_{r>\beta^{1/2}} e^{-cr}rdr,
$$
which entails that $ \|T_1\|\to 0$ as $\beta\to\infty$.

Therefore, we establish the following result:
\[\notag \begin{split}
\|(\chi_+R_\l\chi_+)^j\chi_+R_\l\chi_-R_\l^{k-j-1}\chi_+\|_{{\frak S}_1}\leq \|(\chi_+R_\l\chi_+)^jT_1R_\l^{k-j-1}\chi_+\|_{{\frak S}_1}+\\
\|(\chi_+R_\l\chi_+)^jT_2R_\l^{k-j-1}\chi_+\|_{{\frak S}_1} \leq \|\chi_+R_\l\chi_+\|^j_{{\frak S}_{k-1}} \|T_1\| \|R_\l^{k-j-1}\chi_+\|_{{\frak S}_{(k-1)/(k-j-1)}}\\
+\| \chi_+R_\l\chi_+\|^j_{{\frak S}_{k}}
\|T_2\|_{{\frak S}_{k}}\|R_\l^{k-j-1}\chi_+\|_{{\frak S}_{k/(k-j-1)}}
=o(\beta^{2}),\qquad \text{as } \beta\to\infty.
\end{split}
\]
Substituting this into \eqref{2.3}, we establish \eqref{2.2}. The proof is complete.
\end{proof}

Observe that Proposition ~\ref{pr2.2} implies that
\begin{eqnarray*}
    n_+(\tau,\chi_j(D_m-\lambda I)^{-1}\chi_j)\sim\frac{\beta^2}{4\pi}[((\lambda+\tau^{-1})_+^2-m^2)_+]^{1/2}\text{Area }Q_j.
\end{eqnarray*} Note also that if $\delta>0$ is sufficiently small, then $\Psi(\theta)|x|^{-p}$ can be well approximated by a constant function on the cube $\delta([0,1)^2+n),\forall n\in\mathbb Z^2$. Therefore, we first focus on the case where $V(x)$ is strictly (rather than asymptotically) equal to $\Psi(\theta)|x|^{-p}$, and introduce two points $x_j^+$ and $x_j^-$ in $Q_j$ such that
\begin{eqnarray*}
    V(x_j^+)=\max_{x\in Q_j}V(x)\qquad\text{and }\qquad V(x_j^-)=\min_{x\in Q_j}V(x),
\end{eqnarray*} respectively. Let $x\in\alpha^{1/p}Q_j$, then $\alpha V(x)=\alpha\Psi(\theta)|x|^{-p}=\Psi(\theta)|x\alpha^{-1/p}|^{-p}$. Thus,  $V(x_j^-)\le\alpha V(x)\le V(x_j^+), \forall x\in\alpha^{1/p}Q_j$. Taking the square root on  all sides of this inequality, we obtain $W(x_j^-)\le\sqrt\alpha W(x)\le W(x_j^+)$. Using the monotonicity of $n_+(\t\alpha^{-1},\chi_jW(D_m-\lambda I)^{-1}W\chi_j)$ with respect to $W$, we get
\begin{eqnarray*}
    &&n_+(\t\alpha^{-1},\chi_jW(D_m-\lambda I)^{-1}W\chi_j)=n_+(\t,\sqrt\alpha\chi_jW(D_m-\lambda I)^{-1}W\chi_j\sqrt\alpha)\\
    &&\le n_+(\t,W(x_j^+)\chi_j(D_m-\lambda I)^{-1}\chi_jW(x_j^+))=n_+(\t,V(x_j^+)\chi_j(D_m-\lambda I)^{-1}\chi_j)\\
    &&=n_+\Big(\frac{\t}{ V(x_j^+)},\chi_j(D_m-\lambda I)^{-1}\chi_j\Big)\sim\frac{\beta^2}{4\pi}[((\lambda+\t^{-1}V(x_j^+))^2-m^2)_+]^{1/2}\text{Area }Q_j.
\end{eqnarray*} Consequently, 
\begin{align*}
    \limsup_{\alpha\to\infty}\alpha^{-2/p}N_\tau(\lambda,\alpha)=\limsup_{\alpha\to\infty}\sum_{j=1}^ln_+(\t\alpha^{-1},\chi_jW(D_m-\lambda I)^{-1}W\chi_j)\\
    \le\frac{1}{4\pi}\sum_{j=1}^l[((\lambda+\t^{-1}V(x_j^+))_+^2-m^2)_+]^{1/2}\text{vol}(Q_j),
\end{align*} where $N_\tau(\lambda,\alpha)=n_+(\t\alpha^{-1},W_2(D_m-\lambda I)^{-1}W_2)$. This leads to the following estimate:
\begin{align}
    \limsup_{\alpha\to\infty}\alpha^{-2/p}N_\tau(\lambda,\alpha)\le\frac{1}{4\pi}\int_{\varepsilon_1<|x|<\varepsilon_2}[((\lambda+\t^{-1}\Psi(\theta)|x|^{-p})^2_+-m^2)_+]^{1/2}\,dx.\label{one side}
\end{align} Similarly,
\begin{align}
    \liminf_{\alpha\to\infty}\alpha^{-2/p}N_\tau(\lambda,\alpha)=\liminf_{\alpha\to\infty}\sum_{j=1}^ln_+(\t\alpha^{-1},\chi_jW(D_m-\lambda I)^{-1}W\chi_j)\nonumber\\
    \ge\frac{1}{4\pi}\int_{\varepsilon_1<|x|<\varepsilon_2}[((\lambda+\t^{-1}\Psi(\theta)|x|^{-p})_+^2-m^2)_+]^{1/2}\,dx.\label{another side}
\end{align} Synthesizing inequalities \eqref{one side} and \eqref{another side}, we obtain
\begin{eqnarray}
\label{result for N2}
    \lim_{\alpha\to\infty}\alpha^{-2/p}N_\tau(\lambda,\alpha)=\frac{1}{4\pi}\int_{\varepsilon_1<|x|<\varepsilon_2}[((\lambda+\t^{-1}\Psi(\theta)|x|^{-p})_+^2-m^2)_+]^{1/2}\,dx.
\end{eqnarray} 
This  completes  the analysis of the operator $W_2(D_m-\lambda I)^{-1}W_2$.

\bigskip

\bigskip

{\it The end of the proof of Theorem ~\ref{second theorem}}.
By \eqref{result for the first self adjoint operator}, \eqref{result for W3W3} and \eqref{result for N2}, we deduce
\begin{align}
\label{long estimate}
    \limsup_{\alpha\to\infty}\alpha^{\frac{2}{p}}N(\lambda,\alpha)\leq C_{18}\varepsilon_1^{2-p}\t^{-1}+\frac{1}{4\pi}\int_{\varepsilon_1<|x|<\varepsilon_2}[((\lambda+ \t^{-1}\Psi(\theta)|x|^{-p})_+^2-m^2)_+]^{1/2}\,dx
\end{align} for sufficiently large $\varepsilon_2$ and any $\t\in(0,1)$. Furthermore, if
\begin{eqnarray*}
    \lambda+ \t^{-1}\|\Psi\|_\infty\varepsilon_2^{-p}<m,
\end{eqnarray*} or equivalently,
\begin{eqnarray*}
    \varepsilon_2>\left(\frac{\t^{-1}\|\Psi\|_\infty}{m-\lambda}\right)^{1/p},
\end{eqnarray*} then
\begin{eqnarray*}
    \int_{|x|>\varepsilon_2}[((\lambda+\t^{-1}\Psi(\theta)|x|^{-p})_+^2-m^2)_+]^{1/2}\,dx=0.
\end{eqnarray*} Taking the limit as $\varepsilon_1\to0$, we get
\begin{eqnarray*}
    \limsup_{\alpha\to\infty}\alpha^{-2/p}N(\lambda,\alpha)&\leq&\frac{1}{4\pi}\int_{\mathbb R^2}[((\lambda+\t^{-1}\Psi(\theta)|x|^{-p})_+^2-m^2)_+]^{1/2}\,dx\\
    &=&\frac{1}{4\pi\t^{2/p}}\int_{\mathbb R^2}[((\lambda+\Psi(\theta)|x|^{-p})_+^2-m^2)_+]^{1/2}\,dx.
\end{eqnarray*} Note that the integral on the right-hand side converges for $0<p<2$. Taking the limit as $\t\to1$, we obtain
\begin{eqnarray}
\label{upper limit}
    \limsup_{\alpha\to\infty}\alpha^{-2/p}N(\lambda,\alpha)\le\frac{1}{4\pi}\int_{\mathbb R^2}[((\lambda+\Psi(\theta)|x|^{-p})_+^2-m^2)_+]^{1/2}\,dx.
\end{eqnarray}
Also, for any $\t>1$,
\begin{eqnarray*}
    \liminf_{\alpha\to\infty}N(\lambda,\alpha)\ge\liminf_{\alpha\to\infty}\alpha^{-2/p} N_\tau(\lambda,\alpha).
\end{eqnarray*} Repeating the same steps as for the upper limit, we infer that 
\begin{eqnarray}
\label{lower limit}
    \liminf_{\alpha\to\infty}\alpha^{-2/p}N(\lambda,\alpha)\geq\frac{1}{4\pi}\int_{\mathbb R^2}[((\lambda+\Psi(\theta)|x|^{-p})_+^2-m^2)_+]^{1/2}\,dx.
\end{eqnarray} Combining \eqref{upper limit} and \eqref{lower limit}, we derive the following formula
\begin{eqnarray*}
    N(\lambda,\alpha)\sim\frac{\alpha^{2/p}}{4\pi}\int_{\mathbb R^2}[((\lambda+\Psi(\theta)|x|^{-p})_+^2-m^2)_+]^{1/2}\,dx,\quad\text{as }\alpha\to\infty,
\end{eqnarray*} which finalizes the proof of Theorem ~\ref{second theorem}.$\,\,\,\blacksquare$

\end{document}